\documentclass[12pt,reqno]{amsart}
\usepackage{latexsym,amsmath,amsfonts,amssymb,amsthm}
\def\pmod #1{\ ({\rm{mod}}\ #1)}
\def\Z{\mathbb Z}

\def\C{\mathbb C}
\def\1{{\mathbf 1}}

\def\jacob #1#2{\left(\frac{#1}{#2}\right)}
\def\pmod #1{\ ({\rm{mod}}\ #1)}

\def\floor #1{\left\lfloor{#1}\right\rfloor}

\theoremstyle{plain}
\newtheorem{Thm}{Theorem}
\newtheorem{Lem}{Lemma}

\theoremstyle{definition}
\newtheorem*{Ack}{Acknowledgment}
\theoremstyle{remark}

\begin{document}

\title{On the Atkin and Swinnerton-Dyer type congruences for some truncated hypergeometric ${}_1F_0$ series}
\author{Yong Zhang}
\email{yongzhang1982@163.com}
\address{Department of Mathematics and Physics, Nanjing Institute of Technology,
Nanjing 211167, People's Republic of China}
\author{Hao Pan}
\email{haopan79@zoho.com}
\address{School of Applied Mathematics, Nanjing University of Finance and Economics, Nanjing 210046, People's Republic of China}
\thanks{}
\begin{abstract}
Let $p$ be an odd prime and let $n$ be a positive integer. For any positive integer $\alpha$ and $m\in\{1,2,3\}$, we have
\begin{align*} 
\sum_{k=0}^{p^{\alpha}n-1}\frac{(\frac12)_k}{k!}\cdot\frac{(-4)^k}{m^k}\equiv\bigg(\frac{m(m-4)}{p}\bigg)\sum_{k=0}^{p^{\alpha-1}n-1}\frac{(\frac12)_k}{k!}\cdot\frac{(-4)^k}{m^k}\pmod{p^{2\alpha}},
\end{align*}
where $(x)_k=x(x+1)\cdots(x+k-1)$ and $\jacob{\cdot}{\cdot}$ denotes the Legendre symbol. Also, when $m=4$,
\begin{align*} 
\sum_{k=0}^{p^{\alpha}n-1}(-1)^k\cdot\frac{(\frac12)_k}{k!}\equiv p\sum_{k=0}^{p^{\alpha-1}n-1}(-1)^k\cdot\frac{(\frac12)_k}{k!}\pmod{p^{2\alpha}}.
\end{align*}
\end{abstract}
\maketitle

\section{Introduction}
\setcounter{equation}{0}
\setcounter{Thm}{0}
\setcounter{Lem}{0}
\setcounter{Cor}{0}
\setcounter{Conj}{0}

In \cite{ASD71}, Aktin and Swinnerton-Dyer systematically investigated the arithmetic properties of the Fourier coefficients of noncongruence modular forms. They observed that if $\Gamma$ is a noncongruence subgroup of $SL_2(\Z)$ with a finite index and $k\geq 2$ is even, then
for some good primes $p$, there exists a basis $\{f_i\}_{1\leq i\leq d}$ of $S_k(\Gamma)$, where $d=\dim S_k(\Gamma)$, such that for each $1\leq i\leq d$ and $\alpha\geq 1$,
$$
a_{np^\alpha}(f_i)-\lambda_{p,i}\cdot a_{np^{\alpha-1}}(f_i)+p^{k-1}a_{np^{\alpha-2}}(f_i)\pmod{p^{(k-1)\alpha}},\qquad\forall n\geq1,
$$
where $\lambda_{p,i}$ is an algebraic integer with $|\lambda_{p,i}|\leq 2p^{\frac{k-1}{2}}$, $a_n(f)$ denotes the $n$-th coefficients in the Fourier  expansion of $f(z)$ and $a_x(f)=0$ if $x\not\in\Z$. Subsequently, the work of Aktin and Swinnerton-Dyer was greatly developed by Scholl in \cite{Sc85}.

Nowadays, for a sequence $\{a_n\}_{n\geq 0}$ of integers, the congruence of the form
$$
a_{np^{\alpha}}\equiv \lambda_p\cdot a_{np^{\alpha-1}}\pmod{p^{k\alpha}},\qquad\forall n\geq 1,
$$
is also often called {\it Atkin and Swinnerton-Dyer type congruence}, where $p$ is a prime and $k,r\geq 1$. The Atkin and Swinnerton-Dyer type congruences have be established for many combinatorial sequences. For examples, Beukers \cite{Be85} proved that the Ap\'ery number
$$
A_n:=\sum_{k=0}^n\binom{n}{k}^2\binom{n+k}{k}^2,
$$
which was used to prove the irrationality of $\zeta(3)=\sum_{n\geq 1}n^{-3}$ by Ap\'ery, satisfies the Atkin and Swinnerton-Dyer type congruence
\begin{equation}
A_{np^\alpha-1}\equiv A_{np^{\alpha-1}-1}\pmod{p^{3\alpha}},\qquad\forall n\geq 1,
\end{equation} 
where $p\geq 5$ is prime and $\alpha\geq 1$. Another example due to Coster and Hamme is concerning the Legendre polynomial
$$
P_n(z):=\sum_{k=0}^n\binom{n}{k}\binom{-n-1}{k}\cdot\bigg(\frac{1-z}{2}\bigg)^k.
$$
Coster and Hamme \cite{CoHa91} proved that if the elliptic curve $y^2=x(x^2+Ax+B)$ has the complex multiplication, then the sequence $\{P_n(z)\}_{n\geq 0}$, where $z=(1-A/\sqrt{A^2-4B})/2$, obeys some Atkin and Swinnerton-Dyer type congruences.
In \cite{LiLo14}, Li and Long gave a nice survey on the Atkin and Swinnerton-Dyer congruences. For more related results, the reader may refer to \cite{St14,OSS16}. In particular, recently Sun \cite{ZWS2} proposed many conjectured Atkin and Swinnerton-Dyer type congruences.

On the other hand, define the truncated hypergeometric function
$$
{}_{m+1}F_m\bigg[\begin{matrix}
a_0&a_1&\ldots&a_m\\
&b_1&\ldots&b_m
\end{matrix}\bigg|\,z\bigg]_n:=\sum_{k=0}^{n}\frac{(a_0)_k(a_1)_k\cdots(a_m)_k}{(b_1)_k\cdots(b_m)_k}\cdot\frac{z^k}{k!},
$$
where $$(a)_k=\begin{cases}a(a+1)\cdots(a+k-1),&\text{if }k\geq1,\\ 1,&\text{if }k=0.\end{cases}$$
Clearly the truncated hypergeometric function is just a finite analogue of the original  hypergeometric function. Recently the arithmetic properties of the truncated hypergeometric functions are widely studied. In this paper, we shall consider the simplest truncated hypergeometric function
$$
{}_{1}F_0\bigg[\begin{matrix}
\frac12\\
{}
\end{matrix}\bigg|\,z\bigg]_n=\sum_{k=0}^{n}\frac{(\frac12)_k}{k!}\cdot z^k.
$$

For each non-zero integer $m$, as a consequence of , for each odd prime $p$ we have
\begin{equation}\label{1F0124m}
{}_{1}F_0\bigg[\begin{matrix}
\frac12\\
{}
\end{matrix}\bigg|\,-\frac4m\bigg]_{p-1}\equiv\jacob{m(m-4)}{p}\pmod{p},
\end{equation}
where $\jacob{\cdot}{\cdot}$ denotes the Legendre symbol. In fact, (\ref{1F0124m}) also easily follows from that
\begin{align*}
{}_{1}F_0\bigg[\begin{matrix}
\frac12\\
{}
\end{matrix}\bigg|\,-\frac4m\bigg]_{p-1}\equiv
&{}_{1}F_0\bigg[\begin{matrix}
\frac{1-p}2\\
{}
\end{matrix}\bigg|\,-\frac4m\bigg]_{p-1}=\sum_{k=0}^{\frac{p-1}{2}}\binom{\frac{p-1}{2}}{k}\cdot \bigg(-\frac 4m\bigg)^k\\
=&\bigg(1-\frac4m\bigg)^{\frac{p-1}{2}}\equiv\jacob{m(m-4)}{p}\pmod{p}. 
\end{align*}
In \cite{ST1}, Sun  extended (\ref{1F0124m}) to
\begin{equation}\label{1F0124mp2}
{}_{1}F_0\bigg[\begin{matrix}
\frac12\\
{}
\end{matrix}\bigg|\,-\frac4m\bigg]_{p-1}\equiv\jacob{m(m-4)}{p}+u_{p-\jacob{m(m-4)}{p}}(m-2,1)\pmod{p^2},
\end{equation}
where the Lucas sequence $\{u_n(A,B)\}_{n\geq 0}$ is given by
$$
u_0(A,B)=0,\quad u_1(A,B)=1,\quad u_n(A,B)=Au_{n-1}(A,B)-Bu_{n-2}(A,B),\quad\forall n\geq 2.
$$
Recently, Sun \cite{ZWS2} also obtained an Atkin and Swinnerton-Dyer type generalization of (\ref{1F0124mp2}):
\begin{align}\label{1F0124mpa}
&{}_{1}F_0\bigg[\begin{matrix}
\frac12\\
{}
\end{matrix}\bigg|\,-\frac4m\bigg]_{np^\alpha-1}-\jacob{m(m-4)}{p}{}_{1}F_0\bigg[\begin{matrix}
\frac12\\
{}
\end{matrix}\bigg|\,-\frac4m\bigg]_{np^{\alpha-1}-1}\notag\\
\equiv&\frac{np^{\alpha-1}}{m^{np^{\alpha-1}-1}}\cdot\binom{2np^{\alpha-1}-1}{np^{\alpha-1}-1}\cdot u_{p-\jacob{m(m-4)}{p}}(m-2,1)\pmod{p^{\alpha+1}}.
\end{align}
Clearly (\ref{1F0124mp2}) easily follows from (\ref{1F0124mpa}) by substituting $\alpha=1$ and $n=1$.

It is natural to ask whether in (\ref{1F0124mpa}) modulo $p^{\alpha+1}$ can be replaced by $p^{2\alpha}$. Unfortunately, seemingly it is not easy to get such an extension for general $m$. However, in this paper, for $m=1,2,3$, we shall prove that
\begin{Thm}
Let $p$ be an odd prime and let $n$ be a positive integer. If $m\in\{1,2,3\}$, then for any positive integer $\alpha$, 
\begin{align} \label{jiankun11}
{}_{1}F_0\bigg[\begin{matrix}
\frac12\\
{}
\end{matrix}\bigg|\,-\frac4m\bigg]_{np^\alpha-1}\equiv\jacob{m(m-4)}{p}{}_{1}F_0\bigg[\begin{matrix}
\frac12\\
{}
\end{matrix}\bigg|\,-\frac4m\bigg]_{np^{\alpha-1}-1}\pmod{p^{2\alpha}}.\end{align}
Furthermore, when $m=4$, 
\begin{align} 
{}_{1}F_0\bigg[\begin{matrix}
\frac12\\
{}
\end{matrix}\bigg|\,-1\bigg]_{np^\alpha-1}\equiv p\cdot {}_{1}F_0\bigg[\begin{matrix}
\frac12\\
{}
\end{matrix}\bigg|\,-1\bigg]_{np^{\alpha-1}-1}
\pmod{p^{2\alpha}}\label{jiankun12}.
\end{align}
\end{Thm}
We mention that the special case $m=\alpha=1$ of (\ref{jiankun11}) was also conjectured by  Apagodu and  Zeilberger\cite{AZ} and proved by Liu \cite{L}.

Let us give an explanation on (\ref{jiankun11}) from the viewpoint of convergent series. We know that
\begin{equation}\label{12kzk}
\sum_{k=0}^{n}\frac{(\frac12)_k}{k!}\cdot z^k=\sqrt{1-z}
\end{equation}
for any $z\in\C$ with $|z|\leq 1$. However, since $(\frac12)_k/k!$ is not divisible by 
$p$ for infinitely many $k$, the series (\ref{12kzk}) can't be convergent in the sense of $p$-adic norm. Let $$
{\mathcal S}_n=\sum_{k=0}^{n-1}\frac{(\frac12)_k}{k!}\cdot\frac{(-4)^k}{m^k}.
$$
Then (\ref{jiankun11}) says that for each $n\geq 1$, both $\{{\mathcal S}_{np^{2\alpha}}\}_{\alpha\geq 0}$ and $\{{\mathcal S}_{np^{2\alpha-1}}\}_{\alpha\geq 1}$ are rapidly convergent subsequences of $\{{\mathcal S}_{m}\}_{m\geq0}$ in the sense of $p$-adic norm.

Throughout this paper,
we will show several lemmas  in Sections 2. Theorem 1.1 will be proved in Sections 3.

\section{Some Lemmas}
\setcounter{equation}{0}
\setcounter{Thm}{0}
\setcounter{Lem}{0}
\setcounter{Cor}{0}
\setcounter{Conj}{0}
\begin{Lem}
\label{lzhang}
For any nonnegative integer $k,n$ and $\alpha$, we have

\noindent{\rm (i)}If $p\mid k$, then
\begin{equation}\label{y7}
\binom{p^\alpha n}{k}\equiv \binom{p^{\alpha-1}n}{k/p}\pmod{p^{2\alpha}}.
\end{equation}
\noindent{\rm (ii)}If $p \nmid k$, then
\begin{equation}\label{y81}
\binom{p^\alpha n}{k}\equiv {\frac{p^\alpha n}{k}\binom{p^{\alpha-1}n-1}{\floor {\frac{k-1}{p}}}(-1)^{k-1-\floor {\frac{k-1}{p}}}}\pmod{p^{2\alpha}}.
\end{equation}
\noindent{\rm (iii)}
\begin{equation}\label{y71}
\binom{p^\alpha n-1}{k}\equiv \binom{p^{\alpha-1}n-1}{\floor {k/p}}(-1)^{k-\floor {k/p}}\pmod{p^{\alpha}},
\end{equation}
\end{Lem}
here (\ref{y71}) is the Lemma2(i) in F. Beukers' paper\cite{FB}. 
The following curious identity is due to Sun and Taurso:
\begin{Lem}[{\cite[(2.1)]{ST2}}]
\label{y5}
For any nonzero integer $m$ and positive integer $n$, we have

$$m^{n-1}\sum_{k=0}^{n-1}\frac{1}{m^k}\binom{2k}{k}=\sum_{k=0}^{n-1}\binom{2n}{k}u_{n-k}(m-2,1).$$

\end{Lem}
\begin{Lem} 
Let $p>2$ be a prime. For any nonnegative integer $\alpha$ , $s$ with $\alpha\geq s$, then we have
$$\frac{(m^{p^\alpha-p^{\alpha-1}}-1)}{2p^\alpha}\equiv\frac{(m^{p^s-p^{s-1}}-1)}{2p^s}\pmod{p^{s}}.$$
\end{Lem}
\begin{proof}
\begin{align*}
&\frac{(m^{p^\alpha-p^{\alpha-1}}-1)}{2p^\alpha}=\frac{\sum_{k=1}^{p^{\alpha-s}}\binom{p^{\alpha-s}}{k}(m^{p^s-p^{s-1}}-1)^k}{2p^\alpha}
=\frac{(m^{p^s-p^{s-1}}-1)}{2p^s}\\
+&\frac{(m^{p^s-p^{s-1}}-1)}{2p^\alpha}\bigg(\sum_{k=2}^{p^{\alpha-s}}{}^{'}\binom{p^{\alpha-s}}{k}(m^{p^s-p^{s-1}}-1)^{k-1}+\sum_{k=1}^{p^{\alpha-s-1}}\binom{p^{\alpha-s}}{pk}(m^{p^s-p^{s-1}}-1)^{pk-1}\bigg)\\
\equiv&\frac{(m^{p^s-p^{s-1}}-1)}{2p^s}\pmod{p^{s}}.
\end{align*}
\end{proof}
\begin{Lem} 
Let $p>2$ be a prime. For any nonnegative integer $n$, $l$, $\alpha$ , $s$ and $\alpha\geq s$.  If $m=1, 2, 3,$ then we have

\label{2zhang}
\begin{align}
\label{jiankun5}
&\sum_{\floor{k/p^s}=l}{}^{'}\frac{(-1)^{k}u_{p^{\alpha}n-k}(m-2,1)}{k}\equiv\\
&\bigg(\frac{m(m-4)}{p}\bigg)^s\frac{(-m^{p^\alpha-p^{\alpha-1}}+1)(-1)^l}{2p^\alpha}(u_{p^{\alpha-s}n-l}(m-2,1)+u_{p^{\alpha-s}n-l-1}(m-2,1))\pmod{p^{s}}\notag,
\end{align}
here $\sum_{\floor {k/p^s}=l}{}^{'}$ denotes the sum of $k$ with  $p \nmid k$.
\end{Lem}
\begin{proof}
Let $m$ be an integer . We first assume that the following congruence is right.
\begin{align}
\label{jiankun4}
&\sum_{k=1}^{p^s-1}{}^{'}\frac{(-1)^{k}}{k}u_{p^\alpha n-k}(m-2,1)\\\notag
&\equiv\bigg(\frac{m(m-4)}{p}\bigg)^s\frac{(-m^{p^s-p^{s-1}}+1)}{2p^s}(u_{p^{\alpha-s}n}(m-2,1)+u_{p^{\alpha-s}n-1}(m-2,1))\pmod{p^{s}}.
\end{align}
Note that $u_{-k}(m-2,1)=-u_{k}(m-2,1)$, then
\begin{align*}
&\sum_{\floor {k/p^s}=l}{}^{'}\frac{(-1)^ku_{p^\alpha n-k}(m-2,1)}{k}={\sum_{k=1}^{p^s-1}{}^{'}\frac{(-1)^{p^sl+k}u_{p^\alpha n-p^sl-k}(m-2,1)}{p^sl+k}}\\
&\equiv(-1)^{p^sl}\sum_{k=1}^{p^s-1}{}^{'}\frac{(-1)^{k}u_{p^\alpha n+k-p^sl-p^s}(m-2,1)}{k}\pmod{p^s}\\
&=(-1)^{p^sl+1}\sum_{k=1}^{p^s-1}{}^{'}\frac{(-1)^{k}u_{p^s(l+1-p^{\alpha-s} n)-k}(m-2,1)}{k}.
\end{align*}
Here we take $s, (l+1-p^{\alpha-s} n)$ instead of $\alpha, n$ in (\ref{jiankun4}), then (\ref{jiankun5}) is done.
\begin{align*}
&\sum_{k=1}^{p^s-1}{}^{'}\frac{(-1)^{k}u_{p^s(l+1-p^{\alpha-s} n)-k}(m-2,1)}{k}\\
&\equiv\bigg(\frac{m(m-4)}{p}\bigg)^s\frac{(-m^{p^s-p^{s-1}}+1)}{2p^s}(u_{l+1-p^{\alpha-s} n}(m-2,1)+u_{l-p^{\alpha-s} n}(m-2,1))\pmod{p^{s}}
\end{align*}
Next we will prove (\ref{jiankun4}).
On the one hand, we split the sum into a sum with $p\nmid k$ and one with $k=lp$,
\begin{align}
\label{jiankun7}\sum_{k=1}^{p^s-1}{}^{'}\binom{p^s}{k}u_{p^\alpha n-k}(m-2,1)=\sum_{k=0}^{p^s}\binom{p^s}{k}u_{p^\alpha n-k}(m-2,1)-\sum_{l=0}^{p^{s-1}}\binom{p^s}{pl}u_{p^\alpha n-pl}(m-2,1),
\end{align}
with the help of Lemma \ref{lzhang}(i) and $u_{pl}(m-2,1)=(\frac{m(m-4)}{p})u_l(m-2,1)$, then
\begin{align}
\label{jiankun8}\sum_{l=0}^{p^{s-1}}\binom{p^s}{pl}u_{p^\alpha n-pl}(m-2,1)\equiv{\bigg(\frac{m(m-4)}{p}\bigg)\sum_{k=0}^{p^{s-1}}\binom{p^{s-1}}{k}u_{p^{\alpha-1} n-k}(m-2,1)}\pmod{p^{2s}}.\end{align}
On the other hand, with the help of Lemma \ref{lzhang}(iii), then we get
\begin{align}
\label{jiankun1}
&\notag\sum_{k=1}^{p^s-1}{}^{'}\binom{p^s}{k}u_{p^\alpha n-k}(m-2,1)\equiv\sum_{k=1}^{p^{s}-1}{}^{'}
\frac{p^s}{k}\binom{p^{s-1}-1}{{\floor {\frac{k-1}{p}}}}u_{p^\alpha n-k}(m-2,1)(-1)^{k-1-\floor {\frac{k-1}{p}}}\\
&=p^s\sum_{t=0}^{p^{s-1}-1}\binom{p^{s-1}-1}{t}(-1)^t\sum_{\floor {\frac{k}{p}}=t}{}^{'}\frac{(-1)^{k-1}}{k}u_{p^\alpha n-k}(m-2,1)\pmod{p^{2s}},
\end{align}
we may assume $s\geq 1$ in Lemma \ref{2zhang}. Clearly, we proceed by induction, that for $s=1,2,\ldots,r-1$, (\ref{jiankun5}) is right.
\begin{align}
\notag
&\frac{1}{p^s}\sum_{k=1}^{p^s-1}{}^{'}\binom{p^s}{k}u_{p^\alpha n-k}(m-2,1)\equiv \sum_{t=0}^{p^{s-1}-1}\binom{p^{s-1}-1}{t}(-1)^{t}\bigg(\sum_{\floor {\frac{k}{p}}=t}{}^{'}\frac{(-1)^{k-1}}{k}u_{p^{\alpha}n-k}(m-2,1)\\\notag
-&\bigg(\frac{m(m-4)}{p}\bigg)\frac{(-m^{p^\alpha-p^{\alpha-1}}+1)(-1)^{t}}{2p^\alpha}(u_{p^{\alpha-1}n-t}(m-2,1)+u_{p^{\alpha-1}n-t-1}(m-2,1))\bigg)\\\notag
+&\bigg(\frac{m(m-4)}{p}\bigg)\frac{(-m^{p^\alpha-p^{\alpha-1}}+1)}{2p^\alpha}\sum_{k=0}^{p^{s-1}}\binom{p^{s-1}}{k}u_{p^{\alpha-1}n-k}(m-2,1)\pmod{p^{s}}
\\\notag
\end{align}
We apply Lemma2.1 with $s-1, 1$ instead of $\alpha, n$, then 
\begin{align}
\label{jiankun6}
&\frac{1}{p^s}\sum_{k=1}^{p^s-1}{}^{'}\binom{p^s}{k}u_{p^\alpha n-k}(m-2,1)\equiv \sum_{n_1=0}^{p^{s-2}-1}\binom{p^{s-2}-1}{n_1}(-1)^{n_1}\bigg(\sum_{\floor {k/{p^2}}=n_1}{}^{'}\frac{(-1)^{k-1}}{k}u_{p^{\alpha}n-k}\\\notag
-&\sum_{t=0}^{p-1}\bigg(\frac{m(m-4)}{p}\bigg)\frac{(-m^{p^\alpha-p^{\alpha-1}}+1)(-1)^{pn_1+t}}{2p^\alpha}(u_{p^{\alpha-1}n-pn_1-t}(m-2,1)+u_{p^{\alpha-1}n-pn_1-t-1}(m-2,1))\bigg)\\\notag
+&\bigg(\frac{m(m-4)}{p}\bigg)\frac{(-m^{p^\alpha-p^{\alpha-1}}+1)}{2p^\alpha}\sum_{k=0}^{p^{s-1}}\binom{p^{s-1}}{k}u_{p^{\alpha-1}n-k}(m-2,1)\pmod{p^{s}},
\end{align}
here
\begin{align*}
&\sum_{t=0}^{p-1}(-1)^{pn_1+t}(u_{p^{\alpha-1}n-pn_1-t}(m-2,1)+u_{p^{\alpha-1}n-pn_1-t-1}(m-2,1))\\
=&(-1)^{n_1}(u_{p^{\alpha-1}n-pn_1}(m-2,1)+u_{p^{\alpha-1}n-pn_1-p}(m-2,1))\\
=&(-1)^{n_1}\bigg(\frac{m(m-4)}{p}\bigg)(u_{p^{\alpha-2}n-n_1}(m-2,1)+u_{p^{\alpha-2}n-n_1-1}(m-2,1)),
\end{align*}
Repeat this process $s-1$ times as (\ref{jiankun6}), then 
\begin{align}
\notag
&\frac{1}{p^s}\sum_{k=1}^{p^s-1}{}^{'}\binom{p^s}{k}u_{p^\alpha n-k}(m-2,1)\\\notag&\equiv\ldots\equiv
\sum_{n_1=0}^{p^{}-1}\binom{p^{}-1}{n_1}(-1)^{n_1}\bigg(\sum_{\floor {k/{p^{s-1}}}=n_1}{}^{'}\frac{(-1)^{k-1}}{k}u_{p^{\alpha}n-k}(m-2,1)\\\notag
-&\bigg(\frac{m(m-4)}{p}\bigg)^{s-1}\frac{(-m^{p^\alpha-p^{\alpha-1}}+1)(-1)^{n_1}}{2p^\alpha}(u_{p^{\alpha-s+1}n-n_1}(m-2,1)+u_{p^{\alpha-s+1}n-n_1-1}(m-2,1))\bigg)\\\notag
+&\bigg(\frac{m(m-4)}{p}\bigg)\frac{(-m^{p^\alpha-p^{\alpha-1}}+1)}{2p^\alpha} \sum_{k=0}^{p^{s-1}}\binom{p^{s-1}}{k}u_{p^{\alpha-1}n-k}(m-2,1)
\\\notag
\equiv&\sum_{k=1}^{p^{s}-1}{}^{'}\frac{(-1)^{k-1}}{k}u_{p^{\alpha}n-k}(m-2,1)-\bigg(\frac{m(m-4)}{p}\bigg)^s\frac{(-m^{p^\alpha-p^{\alpha-1}}+1)}{2p^\alpha}(u_{p^{\alpha-s}n}(m-2,1)\\\label{jiankun9}
&+u_{p^{\alpha-s}n-1}(m-2,1))+\bigg(\frac{m(m-4)}{p}\bigg)\frac{(-m^{p^\alpha-p^{\alpha-1}}+1)}{2p^\alpha}\sum_{k=0}^{p^{s-1}}\binom{p^{s-1}}{k}u_{p^{\alpha-1}n-k}(m-2,1)\pmod{p^{s}},
\end{align}
from (\ref{jiankun7}), (\ref{jiankun8}) and (\ref{jiankun9}) , then we only need to prove that 
\begin{align}
\notag
&\frac{1}{p^s}\sum_{k=0}^{p^s}\binom{p^s}{k}u_{p^\alpha n-k}(m-2,1)-\frac{1}{p^s}\bigg(\frac{m(m-4)}{p}\bigg)\sum_{k=0}^{p^{s-1}}\binom{p^{s-1}}{k}u_{p^{\alpha-1}n-k}(m-2,1)\equiv\\\label{yps11}
&\bigg(\frac{m(m-4)}{p}\bigg)\frac{(m^{p^\alpha-p^{\alpha-1}}-1)}{2p^\alpha}\sum_{k=0}^{p^{s-1}}\binom{p^{s-1}}{k}u_{p^{\alpha-1}n-k}(m-2,1)\pmod{p^{s}}.
\end{align}
Substitue $m=1$ in (\ref{yps11}).
Then  
\begin{align*}
&\sum_{k=0}^{p^s}\binom{p^s}{k}u_{p^\alpha n-k}(-1,1)-\bigg(\frac{-3}{p}\bigg)\sum_{k=0}^{p^{s-1}}\binom{p^{s-1}}{k}u_{p^{\alpha-1}n-k}(-1,1)\\\notag
&=(-1)^{p^s} u_{p^\alpha n+p^s}(-1,1)
-\bigg(\frac{-3}{p}\bigg)(-1)^{p^{s-1}} u_{p^{\alpha-1} n+p^{s-1}}(-1,1)=0\notag,
\end{align*}
we are done.
Because 
\begin{align}
\notag
&\sum_{k=0}^{p^{s-1}}\binom{p^{s-1}}{k}u_{k-p^{\alpha-1}n}(0,1)
=\frac{(1+i)^{p^{s-1}}i^{-p^{\alpha-1}n}-(1-i)^{p^{s-1}}(-i)^{-p^{\alpha-1}n}}{2i}\\ \label{yps2}
=&2^{\frac{p^{s-1}-1}{2}}(u_{\frac{p^{s-1}-(\frac{-1}{p})^{s-1}}{2}-p^{\alpha-1}n}(0,1)+(\frac{-1}{p})^{s-1}u_{\frac{p^{s-1}-(\frac{-1}{p})^{s-1}}{2}-p^{\alpha-1}n+1}(0,1))
\end{align}
Next we will take $m=2$ in (\ref{yps11}). By (\ref{yps2}), then it suffices to show that 
\begin{align}
\notag
&\frac{2^{p^\alpha-p^{\alpha-1}}-1}{2p^{\alpha-s}}\equiv
\bigg(\bigg(\frac{-1}{p}\bigg)2^{\frac{p^s-p^{s-1}}{2}}\times
\\&\frac{(u_{\frac{p^{s}-(\frac{-1}{p})^{s}}{2}-p^{\alpha}n}(0,1)+(\frac{-1}{p})^{s}u_{\frac{p^{s}-(\frac{-1}{p})^{s}}{2}-p^{\alpha}n+1}(0,1))}{(u_{\frac{p^{s-1}-(\frac{-1}{p})^{s-1}}{2}-p^{\alpha-1}n}(0,1)+(\frac{-1}{p})^{s-1}u_{\frac{p^{s-1}-(\frac{-1}{p})^{s-1}}{2}-p^{\alpha-1}n+1}(0,1))}    -1\bigg)\\\notag
&\equiv 2^{\frac{p^s-p^{s-1}}{2}}(-1)^{\frac{p-(\frac{-1}{p})}{4}}-1
\pmod{p^{s}},
\end{align}
here
\begin{align}
\notag
&(\frac{-1}{p})u_{\frac{p^{s-1}-(\frac{-1}{p})^{s-1}}{2}-p^{\alpha-1}n}(0,1)=u_{\frac{p^{s}-p(\frac{-1}{p})^{s-1}}{2}-p^{\alpha}n}(0,1)
=(-1)^{\frac{p-(\frac{-1}{p})}{4}}u_{\frac{p^{s}-(\frac{-1}{p})^{s}}{2}-p^{\alpha}n}(0,1)\\\notag.
\end{align}
and 
\begin{align}
\notag
&(\frac{-1}{p})u_{\frac{p^{s-1}-(\frac{-1}{p})^{s-1}}{2}-p^{\alpha-1}n+1}(0,1)=u_{\frac{p^{s}-p(\frac{-1}{p})^{s-1}}{2}-p^{\alpha}n+p}(0,1)
=(\frac{-1}{p})(-1)^{\frac{p-(\frac{-1}{p})}{4}}u_{\frac{p^{s}-(\frac{-1}{p})^{s}}{2}-p^{\alpha}n+1}(0,1)\\\notag,
\end{align}
At last, with the help of Lemma 2.3 and the following congruence
\begin{align}
\notag
&\frac{2^{p^s-p^{s-1}}-1}{2p^s}=\frac{1}{2p^s}[(2^{\frac{p^s-p^{s-1}}{2}}(-1)^{\frac{p-(\frac{-1}{p})}{4}}-1)^2+2(2^{\frac{p^s-p^{s-1}}{2}}(-1)^{\frac{p-(\frac{-1}{p})}{4}}-1)]\\
\equiv&\frac{1}{p^s}(2^{\frac{p^s-p^{s-1}}{2}}(-1)^{\frac{p-(\frac{-1}{p})}{4}}-1)
\pmod{p^{s}},
\end{align}
Lemma 2.4 with $m=2$ is concluded because
$$2^{\frac{p^s-p^{s-1}}{2}}\equiv\bigg(\frac{2}{p}\bigg)=(-1)^\frac{p^2-1}{8}=(-1)^{\frac{p-(\frac{-1}{p})}{4}}
\pmod{p^{s}}.$$
When  $m=3,$ (\ref{yps11}) can be proved similarly, with the help of Lemma 2.3, then we have 
\begin{align}
\notag
&\bigg(\frac{-3}{p}\bigg)\frac{(-3)^\frac{p^{s}-p^{s-1}}{2}}{p^{s}}\frac{(\frac{1+\sqrt{3}i}{2})^{p^{\alpha}n+p^{s}}+(\frac{1-\sqrt{3}i}{2})^{p^{\alpha}n+p^{s}}}{(\frac{1+\sqrt{3}i}{2})^{p^{\alpha-1}n+p^{s-1}}+(\frac{1-\sqrt{3}i}{2})^{p^{\alpha-1}n+p^{s-1}}}-\frac{1}{p^{s}}\\\label{jiankun10}
&\equiv\frac{\bigg(\frac{3}{p}\bigg)3^\frac{p^{s}-p^{s-1}}{2}-1}{p^s}\equiv\frac{3^{p^\alpha-p^{\alpha-1}}-1}{2p^\alpha}\pmod{p^{s}},
\end{align}
where
\begin{align*}
&\frac{(\frac{1+\sqrt{3}i}{2})^{p^{\alpha}n+p^{s}}+(\frac{1-\sqrt{3}i}{2})^{p^{\alpha}n+p^{s}}}{(\frac{1+\sqrt{3}i}{2})^{p^{\alpha-1}n+p^{s-1}}+(\frac{1-\sqrt{3}i}{2})^{p^{\alpha-1}n+p^{s-1}}}\\
=&\frac{(\frac{1+\sqrt{3}i}{2})^{(\frac{-3}{p})(p^{\alpha-1}n+p^{s-1})}+(\frac{1-\sqrt{3}i}{2})^{(\frac{-3}{p})(p^{\alpha-1}n+p^{s-1})}}{(-1)^{\frac{p-(\frac{-3}{p})}{3}}[(\frac{1+\sqrt{3}i}{2})^{p^{\alpha-1}n+p^{s-1}}+(\frac{1-\sqrt{3}i}{2})^{p^{\alpha-1}n+p^{s-1}}]}=1,
\end{align*}
(\ref{jiankun10}) is proved when $m=3$.
\end{proof}

\begin{Lem}
\label{3zhang}
Let $a_{k}\in Z_{p}(k=0,1, \ldots)$ be such that 
$$\sum_{\floor {k/p^s}=l}a_{k}\equiv 0\pmod{p^s},
$$for any nonnegative integer $m$ , $n$, $\alpha$ and $s$.
Then 
\begin{equation}\label{zhangyong}
\sum_{\floor {k/p^\alpha}=l}a_{k}\binom{mp^\alpha n-1}{k}(-1)^k\equiv 0\pmod{p^\alpha}.
\end{equation}
\end{Lem}
\begin{proof}
We prove it by induction on $\alpha$. The above congruence is trivial when $\alpha=0,1$.
Suppose that we have show it for $0,1,\ldots,\alpha-1$.
\begin{align}
\sum_{\floor {k/p^\alpha}=l}a_{k}\binom{mp^\alpha n-1}{k}(-1)^k&\equiv \sum_{\floor {k/p^\alpha}=l}a_{k}\binom{mp^{\alpha-1}n-1}{\floor{k/p}}(-1)^{\floor{k/p}}\\\notag
&=\sum_{\floor {t/p^{\alpha-1}}=l}(\sum_{\floor{k/p}=t}a_{k})\binom{mp^{\alpha-1}n-1}{t}(-1)^{t},
\end{align}
we now apply the induction hypothesis for $\alpha-1$ with the new coefficients
$$p\hat{a_{t}}=\sum_{\floor{k/p}=t}a_{k}\equiv{0}\pmod{p},$$
and $$\sum_{\floor {t/p^{\alpha-1}}=l}\hat{a_{t}}=\frac{1}{p}\sum_{\floor {k/p^{\alpha}}=l}a_{k}\equiv{0}\pmod{p^{\alpha-1}}.$$
So we obtain 
$$\sum_{\floor {k/p^\alpha}=l}a_{k}\binom{mp^\alpha n-1}{k}(-1)^k\equiv p\sum_{\floor {t/p^{\alpha-1}}=l}\hat{a_{t}}\binom{mp^{\alpha-1}n-1}{t}(-1)^{t}\equiv{0}\pmod{p^{\alpha}}.$$
\end{proof}
\section{Proofs of Theorem 1.1}
\setcounter{equation}{0}
\setcounter{Thm}{0}
\setcounter{Lem}{0}
\setcounter{Cor}{0}
\setcounter{Conj}{0}
\begin{proof}
According to Lemma2.1(i)  and  \ref{y5}, so (\ref{jiankun11}) can be rewritten as
\begin{align}
\notag
&n\sum_{k=1}^{p^\alpha n-1}{}^{'}\binom{2p^\alpha n-1}{k-1}\frac{u_{p^\alpha n-k}(m-2,1)}{k}\\\label{y9}
&\equiv\bigg(\frac{m(m-4)}{p}\bigg)\frac{(m^{(p^\alpha-p^{\alpha-1})n}-1)}{2p^\alpha}\sum_{k=0}^{p^{\alpha-1}n-1}\binom{2p^{{\alpha-1}}n}{k}u_{p^{\alpha-1} n-k}(m-2,1)\pmod{p^{\alpha}}.
\end{align}
\noindent{\rm (i)}When $m=1$, we need only to prove 
\begin{equation}\label{y10}
\sum_{k=1}^{p^\alpha n-1}{}^{'}\binom{2p^\alpha n-1}{k-1}\frac{u_{p^\alpha n-k}(-1,1)}{k}
\equiv{0}\pmod{p^{\alpha}}.
\end{equation}
However
\begin{align}\notag
&\sum_{k=1}^{p^\alpha n-1}{}^{'}\binom{2p^\alpha n-1}{k-1}\frac{u_{p^\alpha n-k}(-1,1)}{k}
\equiv\sum_{k=1}^{p^\alpha n-1}{}^{'}\binom{2p^\alpha n-1}{k}(-1)^k\frac{(-1)^{k-1}u_{p^\alpha n-k}(-1,1)}{k}\pmod{p^{\alpha}}.
\end{align}
We set $m=2$ and $a_{k}=\frac{(-1)^{k-1}u_{p^\alpha n-k}(-1,1)}{k}$ if $p\nmid k,$ $a_{k}=0$ otherwise in Lemma \ref{3zhang}. Thus (\ref{jiankun11}) with $m=1$ immediately follows from
Lemma \ref{2zhang}.

\noindent{\rm (ii)}
Next we will prove it when $m=2, 3.$
 It suffices to prove that
\begin{align*}
\notag
&n\sum_{k=1}^{p^\alpha n-1}{}^{'}\binom{2p^\alpha n-1}{k-1}\frac{u_{p^\alpha n-k}(m-2,1)}{k}
\equiv\bigg(\frac{m(m-4)}{p}\bigg)\frac{m^{(p^\alpha-p^{\alpha-1})n}-1}{2p^\alpha}\\\label{1y9}
&\sum_{k=0}^{p^{\alpha-1}n-1}\bigg(\binom{2p^{{\alpha-1}}n-1}{k}+\binom{2p^{{\alpha-1}}n-1}{k-1}\bigg)u_{p^{\alpha-1}n-k}(m-2,1)
\equiv n\bigg(\frac{m(m-4)}{p}\bigg)\\\notag
&\frac{-m^{p^\alpha-p^{\alpha-1}}+1}{2p^\alpha}\sum_{k=0}^{p^{\alpha-1}n-1}\binom{2p^{{\alpha-1}}n-1}{k}(u_{k-p^{\alpha-1}n}(m-2,1)+u_{k+1-p^{\alpha-1}n}(m-2,1))\pmod{p^{\alpha}},
\end{align*}
where
\begin{align*}
\frac{m^{(p^\alpha-p^{\alpha-1})n}-1}{2p^\alpha}=\frac{1}{2p^\alpha}\sum_{k=1}^{n}\binom{n}{k}(m^{p^\alpha-p^{\alpha-1}}-1)^k\equiv{n\frac{(m^{p^\alpha-p^{\alpha-1}}-1)}{2p^\alpha}}\pmod{p^{\alpha}}.
\end{align*}
By Lemma \ref{lzhang}, then
\begin{align*}
&\sum_{n_1=0}^{p^{\alpha-1}n-1}\binom{2p^{\alpha-1}n-1}{n_1}(-1)^{n_1}\bigg(\sum_{\floor{k/p}=n_1}{}^{'}\frac{(-1)^{k}u_{k-p^{\alpha}n}(m-2,1)}{k}-
\bigg(\frac{m(m-4)}{p}\bigg)\\
&\frac{(-m^{p^\alpha-p^{\alpha-1}}+1)(-1)^{n_1}}{2p^\alpha}
(u_{{n_1}-p^{\alpha-1}n}(m-2,1)+u_{{n_1}+1-p^{\alpha-1}n}(m-2,1))\bigg)\equiv{0}\pmod{p^{\alpha}}.
\end{align*}
With the help of Lemma 2.4 with $s=1$, we have 
\begin{align*}
&\sum_{n_{1}=0}^{p^{\alpha-2}n-1}\binom{2p^{\alpha-2}n-1}{n_{1}}(-1)^{n_{1}}\bigg(\sum_{\floor{k/p^2}=n_{1}}{}^{'}\frac{(-1)^{k}u_{k-p^{\alpha}n}(m-2,1)}{k}-
\bigg(\frac{m(m-4)}{p}\bigg)^2\\
&\frac{(-m^{p^\alpha-p^{\alpha-1}}+1)}{2p^\alpha}(-1)^{n_{1}}(u_{n_{1}-p^{\alpha-2}n}(m-2,1)+u_{n_{1}+1-p^{\alpha-2}n}(m-2,1))\equiv{0}\pmod{p^{\alpha}},
\end{align*}
then 
\begin{align*}
&\sum_{n_{1}=0}^{p^{\alpha-2}n-1}\binom{2p^{\alpha-2}n-1}{n_{1}}(-1)^{n_{1}}\bigg(\sum_{\floor{k/p^2}=n_{1}}{}^{'}\frac{(-1)^{k}u_{k-p^{\alpha}n}(m-2,1)}{k}-
\bigg(\frac{m(m-4)}{p}\bigg)^2\\
&\frac{(-m^{p^\alpha-p^{\alpha-1}}+1)(-1)^{n_{1}}}{2p^\alpha}(u_{n_{1}-p^{\alpha-2}n}(m-2,1)+u_{n_{1}+1-p^{\alpha-2}n}(m-2,1))\bigg)\equiv{0}\pmod{p^{\alpha}},
\end{align*}
repeat this process $\alpha$ times,  then we obtain
\begin{align*}
&\sum_{t=0}^{n-1}\binom{2n-1}{t}(-1)^t\bigg(\sum_{\floor{k/p^{\alpha}}=0}{}^{'}\frac{(-1)^{t+k}u_{k-p^{\alpha}(n-t)}(m-2,1)}{k}-
\bigg(\frac{m(m-4)}{p}\bigg)^{\alpha-1}\\
&\frac{(-m^{p^\alpha-p^{\alpha-1}}+1)(-1)^{pt}}{2p^\alpha}\sum_{s=0}^{p-1}(-1)^s(u_{pt+s-pn}(m-2,1)+u_{pt+s+1-pn}(m-2,1))\bigg)\\
&=\sum_{t=0}^{n-1}\binom{2n-1}{t}\bigg(\sum_{\floor{k/p^{\alpha}}=0}{}^{'}\frac{(-1)^{k}u_{k-p^{\alpha}(n-t)}(m-2,1)}{k}-
\bigg(\frac{m(m-4)}{p}\bigg)^{\alpha}\\
&\frac{(-m^{p^\alpha-p^{\alpha-1}}+1)}{2p^\alpha}(u_{t-n}(m-2,1)+u_{t+1-n}(m-2,1))\bigg)\equiv0\pmod{p^{\alpha}},
\end{align*}
where the last step we used Lemma 2.4 with $l=0$.
\end{proof}

\begin{Ack}
We are grateful to Professor Zhi-Wei Sun for his helpful discussions on this paper.
\end{Ack}

\end{document}